\documentclass{amsart}

\usepackage{amssymb}
\usepackage{subfigure,tikz}
\usepackage{hyperref}
\usepackage[T1]{fontenc}

\newtheorem{theorem}{Theorem}[section]
\newtheorem{remark}[theorem]{Remark}
\newtheorem{corollary}[theorem]{Corollary}

\newtheorem{proposition}[theorem]{Proposition}
\newtheorem{definition}[theorem]{Definition}
\newtheorem{openproblem}[theorem]{Open Problem}

\newcommand{\eps}{\varepsilon}
\newcommand\Prob{{\mathcal P}}
\newcommand\cOmega{\overline{\Omega}}
\newcommand\sfera{{{\mathbb S}^1}}
\newcommand\cil{\cOmega\times\sfera}

\newcommand{\R}{\mathbb{R}} 
\newcommand{\N}{\mathbb{N}}
\newcommand{\Si}{\Sigma}
\newcommand{\Om}{\Omega}

\newcommand{\Ots}{\overline{\Omega}\times\mathbb{S}^1}

\newcommand\Qt{{\mathcal Q}_t}

\newcommand{\vari}{\mathcal{V}_1(\overline{\Omega})}
\newcommand{\weak}{\rightharpoonup}
\newcommand{\cH}{\mathcal{H}^1}
\newcommand{\cHa}{\mathcal{H}^1_{A}}
\newcommand{\Flp}{F_L}
\newcommand{\Fp}{F_\infty}
\newcommand{\sobn}{C}
\newcommand{\spq}{c_{p,q}}

\DeclareMathOperator*{\esssup}{ess\,sup}

\newdimen\mex
\makeatletter
\def\niv{\mathrel{\hbox{\hglue -0.4\mex
\vrule \@height 1.4\mex \@width .14\mex
\vrule \@height .14\mex \@width .75\mex
\hglue -0.2\mex}}}
\makeatother

\newcommand{\rectangleA}{

\draw[-, line width=0.7pt] (0,0)--(0,10); 
\draw[-, line width=0.7pt] (0,10)--(10,10);
\draw[-, line width=0.7pt] (10,10)--(10,0);
\draw[-, line width=0.7pt] (10,0)--(0,0);
\draw[-, line width=0.7pt] (0,2)--(10,2);
\draw[-, line width=0.7pt] (0,7)--(10,7);

\draw[-,very thin] (0,9.3)--(0.5,10);
\draw[-,very thin] (1,7)--(3,10);
\draw[-,very thin] (3.5,7)--(5.5,10);
\draw[-,very thin] (6,7)--(8,10);
\draw[-,very thin] (8.5,7)--(10,9.2);

\draw[-,very thin] (0,3)--(2.375,2);
\draw[-,very thin] (0,4.5)--(5.94,2);
\draw[-,very thin] (0,6)--(9.5,2);
\draw[-,very thin] (1.187,7)--(10,3.29);
\draw[-,very thin] (4.75,7)--(10,4.79);
\draw[-,very thin] (8.31,7)--(10,6.29);

\draw[-,very thin] (0,1.5)--(1.5,0);
\draw[-,very thin] (1.5,2)--(3.5,0);
\draw[-,very thin] (3.5,2)--(5.5,0);
\draw[-,very thin] (5.5,2)--(7.5,0);
\draw[-,very thin] (7.5,2)--(9.5,0);
\draw[-,very thin] (9.5,2)--(10,1.5);
}

\newcommand{\rectangleB}{

\draw[-, line width=0.7pt] (0,0)--(0,10); 
\draw[-, line width=0.7pt] (0,10)--(10,10);
\draw[-, line width=0.7pt] (10,10)--(10,0);
\draw[-, line width=0.7pt] (10,0)--(0,0);
\draw[-,line width=0.7pt] (0,1)--(10,1);
\draw[-,line width=0.7pt] (0,4)--(10,4);
\draw[-,line width=0.7pt] (0,8)--(10,8);

\draw[-,very thin] (0,0)--(0,1);
\draw[-,very thin] (2,0)--(2,1);
\draw[-,very thin] (4,0)--(4,1);
\draw[-,very thin] (6,0)--(6,1);
\draw[-,very thin] (8,0)--(8,1);

\draw[-,very thin] (0,1.5)--(10,1.5);
\draw[-,very thin] (0,2)--(10,2);
\draw[-,very thin] (0,2.5)--(10,2.5);
\draw[-,very thin] (0,3)--(10,3);
\draw[-,very thin] (0,3.5)--(10,3.5);

\draw[-,very thin] (0,4)--(1,8);
\draw[-,very thin] (2,4)--(3,8);
\draw[-,very thin] (4,4)--(5,8);
\draw[-,very thin] (6,4)--(7,8);
\draw[-,very thin] (8,4)--(9,8);

\draw[-,very thin] (0.5,8)--(0,9);
\draw[-,very thin] (3.5,8)--(2,10);
\draw[-,very thin] (6.5,8)--(5,10);
\draw[-,very thin] (9.5,8)--(8,10);
}

\title[Dirichlet conditions in Poincar\'e-Sobolev inequalities]{Dirichlet conditions in Poincar\'e\,-\,Sobolev inequalities: the sub\,-\,homogeneous case}

\author{Davide Zucco}
\address{Davide Zucco, Istituto Nazionale di Alta Matematica, Unit\`a di Ricerca del Dipartimento di Scienze Matematiche, Politecnico di Torino, Italy}
\email{davide.zucco@polito.it}

\begin{document}

\begin{abstract}
We investigate the dependence of optimal constants in Poincar\'e-Sobolev inequalities of planar domains on the region where the Dirichlet condition is imposed. More precisely, we look for the best Dirichlet regions, among closed and connected sets with prescribed total length $L$ (one-dimensional Hausdorff measure), that make these constants as small as possible. We study their limiting behaviour, showing, in particular, that Dirichlet regions homogenize inside the domain with comb-shaped structures, periodically distribuited at different scales and with different orientations. To keep track of these information we rely on a $\Gamma$-convergence result in the class of varifolds. This also permits applications to reinforcements of anisotropic elastic membranes. At last, we provide some evidences for a conjecture.
\end{abstract}

\maketitle

\section{Introduction}

We investigate the dependence of optimal constants in Poincar\'e-Sobolev inequalities of planar domains on the region where the Dirichlet condition is imposed. Precisely, we assume that are given:
\begin{itemize}
\item[-] two real numbers $p,q$ such that $1<p<+\infty$ and $1\leq q<p$;
\item[-] a bounded domain $\Omega\subset\R^2$ with a Lipschitz boundary $\partial\Omega$;
\item[-] a scalar continuous function $f$ positive over $\overline \Omega$;
\item[-] a $2\times 2$ symmetric, positive definite, matrix-valued function $A(x)=[a_{ij}(x)]$, defined for every $x\in\overline{\Omega}$ and continuous there. This implies 
in particular that $A(x)$ is \emph{uniformly elliptic}, i.e. 
\begin{equation}
\label{uniform}
\kappa_0 |y|^2\leq A(x) y\cdot y \leq \kappa_1|y|^2\quad
\forall y\in\R^2,\quad\forall x\in\overline{\Omega},
\end{equation}
for some constants $0<\kappa_0\leq \kappa_1$. For every $x\in \overline \Omega$ we denote by $a_{\min}(x)$ and $a_{\max}(x)$ the corresponding eigenvalues of $A(x)$.
\end{itemize}
Within this framework, for every set $\Sigma\subset \overline \Omega$ with \emph{positive $p$-capacity} the
following \emph{Poincar\'e-Sobolev inequality} holds (see, for instance, \cite[Chapter~10]{maz} and \cite[Corollary~4.5.2]{zie}): there exists a constant $C>0$ (possibly depending on $p$, $q$, $\Omega$, $f$, $A$ and $\Sigma$) such that
\begin{equation}\label{sobolev}
\Big(\int_{\Om} f(x) u(x)^q dx\Big)^{\frac{p}{q}} \leq C\int_{\Om} |A(x)\nabla u(x)\cdot \nabla u(x)|^{\frac{p}{2}}dx, \quad \text{for all $u\in W_\Sigma(\Omega)$},
\end{equation}
where the class of functions
\[
W_\Sigma(\Omega):=\{u\in W^{1,p}(\Omega)\,| \, u= 0 \text{ on } \Si\}.
\]
The assumption that $\Sigma$ has positive $p$-capacity makes the \emph{Dirichlet condition} along $\Sigma$ meaningful in $W^{1,p}(\Omega)$. It is necessary to anchor the functions somewhere inside $\overline\Omega$, otherwise \eqref{sobolev} would be false, violated by constant functions.
Then the smallest constant $C$ that one can use in \eqref{sobolev} can be defined as 
\begin{equation}
\label{PS}
\sobn(\Omega,\Sigma)
:= \max_{u\in W_\Sigma(\Omega)\setminus \{0\}} \frac{\left(\int_{\Om}f(x)|u(x)|^q dx\right)^{\frac pq}}{\int_{\Om} |A(x)\nabla u(x)\cdot\nabla u(x)|^{\frac{p}{2}}dx}.
\end{equation}
Henceforth we will refer to \eqref{PS} as the \emph{Poincar\'e-Sobolev constant} and we will focus on its \emph{shape} dependence; this is the reason why in \eqref{PS} we just emphasize the domain $\Omega$ and the Dirichlet region $\Sigma$. Moreover, $C(\Omega)$ will stand for $C(\Omega,\partial \Omega)$ (i.e., $\Sigma=\partial \Omega$), whenever this will do not give rise to misunderstandings. Notice  that when $\Sigma=\partial \Omega$, then $W_\Sigma(\Omega)=W_0^{1,p}(\Omega)$, namely the Sobolev space with zero trace condition along the boundary $\partial\Omega$.

Via the Poincar\'e-Sobolev inequality \eqref{sobolev} one can prove the existence of a maximizer for \eqref{PS} solving (in a weak sense) the following \emph{anisotropic $p$-Laplace} (see \cite{alesig}) nonlinear eigenvalue problem 
\begin{equation*}
\begin{cases}
-\mathrm{div}(|A\nabla u\cdot \nabla u|^{\frac{p-2}{2}}A\nabla u)={\lambda}f\|u\|_{q}^{p-q} |u|^{q-2}u,\quad &\text{in $\Omega$}\\
u=0  \quad &\text{on $\Sigma$}
\end{cases}
\end{equation*}
with $\lambda=1/C(\Omega,\Sigma)$. Notice that, if for a certain $\lambda$ there exist non-trivial solutions of the previous system, then there must hold $\lambda\geq 1/C(\Omega,\Sigma)$. These considerations suggest that $C(\Omega,\Sigma)$ is a \emph{principal period}, i.e., the reciprocal {principal frequency}  of the system (see \cite{fralam} for an analysis of these nonlinear eigenvalue problems and \cite{bra1,bra2,bra3,mazzuc} for related shape optimization problems). The label \emph{sub-homogeneus} then comes from the fact that when $q<p$ the principal frequency rescales with a power under the \emph{homogeneous} case $p=q$ (see Remark~\ref{r.scaling} below).

For fixed $\Omega$, we look for the best shape and location of the {Dirichlet region} $\Sigma$ that makes as small as possible the Poincar\'e-Sobolev constant \eqref{PS}. Contrary to more classical shape optimization problems, where the Dirichlet condition is prescribed along the boundary $\partial \Omega$, here the Dirichlet region $\Sigma$ is \emph{decoupled} from the domain $\Omega$. The variational problem that we introduce is then the following:
given  $L>0$
\begin{equation}\label{problem}
 \min \left\{\sobn(\Om,\Sigma): \; \text{$\Sigma\subset\overline \Omega$ closed, connected, and $\cH(\Sigma)\leq L$}\right\},
\end{equation}
where $\cH$ denotes the one-dimensional Hausdorff measure. 

A remarkable application to the reinforcement of an anisotropic elastic membrane is obtained by letting $q=1$ in \eqref{PS} (usually, in this case, the Poincar\'e-Sobolev constant $C(\Omega,\Sigma)$ is called \emph{$p$-torsional rigidity}).
It is well known (see, e.g., \cite[Proposition 2.2]{bra1}) that when $q=1$ the quantity $C(\Omega,\Sigma)$ can be written as the $(p-1)$-power of the $p$-\emph{compliance}: 
\[
\sobn(\Omega,\Sigma)=\left(\int_\Omega f(x) u_\Sigma(x) dx\right)^{p-1},
\]
where the state function $u_\Sigma$ is the unique solution of the variational problem 
\[
\min_{u\in W_\Sigma(\Omega)}  \int_{\Om} |A(x)\nabla u(x)\cdot\nabla u(x)|^{\frac{p}{2}}dx- p \int_{\Om}f(x)u(x) dx.
\]
Therefore, we have a model for the deflection of a structure (e.g. a membrane over $\Omega$ with an anisotropic Young modulus driven by the matrix $A$) subjected to a force $f$ acting in the vertical direction, and glued to the ground along $\Sigma$. Indeed, the $p$-compliance represents the work done by the force $f$ at the equilibrium: hence, the {smaller} the compliance the more rigid the structure. This makes physically interesting the study of \eqref{problem} (observe that this problem has been studied in \cite{butsan, nay1, nay2, til} when $A$ is the identity matrix).

The additional constraints imposed on $\Sigma$ are typical in the so called \emph{average distance problems} (see for instance \cite{buoust,butste,leme,mipast,mostil}) and are important in existence results. Indeed, the compactness of the space made up of compact, connected sets with prescribed $\cH$ measure, in the Hausdorff metric, is  entailed by classical results of Blaschke and Go{\l}\k{a}b (see, for instance, \cite[Theorems 4.4.15 and 4.4.17]{ambtil}). Moreover, the lower semicontinuity of the map
$\Sigma\mapsto \sobn(\Omega,\Sigma)$ in the Hausdorff metric is standard routine (see, e.g., \cite[Lemma 5.2]{bra2} in the case $A$ is the identity and $f$ is constant).
These considerations allow to infer the existence of a minimizer in \eqref{problem} for every $L>0$.
Moreover, every minimizer $\Sigma$ satisfies $\cH(\Sigma)=L$ (otherwise if $\cH(\Sigma)<L$, then one could decrease
$C(\Si,\Om)$ by attaching to $\Sigma$ some short
segments, contradicting the optimality, cf.~\cite{buoust,butsan}). Therefore, the larger $L$ the longer the minimizers of \eqref{problem}. This makes interesting questions concerning the asymptotic behaviour of the minimizers:
\begin{quote}
\emph{with what limit density and local orientation will a minimizer of \eqref{problem} distribuite inside $\overline\Omega$, as $L\to+\infty$?}
\end{quote}

\smallskip

The plan of the paper is the following. The main result is stated in Section~\ref{sec1b}, answering the previous question. Section~\ref{sec2} is a collection of some new properties of Poincar\'e-Sobolev constants. Sections~\ref{sec3} and \ref{sec4} then contain, respectively, the $\Gamma$-liminf inequality and the $\Gamma$-limsup inequality.  The last Section~\ref{sec5} concerns the homogeneous and super-homogenous cases (i.e., the case $q\geq p$) and a challenging open problem is provided.

\section{Main result}\label{sec1b}
Via $\Gamma$-convergence theory we analyze the question of the introduction and prove that the minimizers of \eqref{problem} homogenize inside $\overline\Omega$, at different scales and with different orientations, as $L\to+\infty$.   
To keep track on the limit density and on the local orientation of a set we rely on an idea developed in  
\cite{tilzuc2}: to every $\Sigma$ in the class
\begin{equation*}
\mathcal A_L(\Omega):=\left\{\Sigma\subset\overline{\Omega}\,\,:\,\,\text{$\Sigma$ is a closed, connected set with $\cH(\Sigma)\leq L$}\right\},
\end{equation*}
for some $L>0$, we associate the probability measure $\theta_\Sigma$ of the
{product space} $\overline{\Omega}\times
 \mathbb{S}^1$, defined by 
\begin{equation}
\label{concrete}
\int_{\Ots} \varphi(x,y)d\theta_\Sigma(x,y):=\frac{1}{\cH(\Si)}
\int_{\Si}\varphi\big(x,\xi_{\Si}(x)\big)\,d\cH(x),\quad
\varphi\in \mathcal C_{\text{sym}}(\overline{\Omega}\times
 \mathbb{S}^1),
\end{equation}
where $\xi_\Sigma(x)$ is the unit normal vector to $\Sigma$ at $x\in\Sigma$, while
\[
\mathcal C_{\text{sym}}(\overline{\Omega}\times {\mathbb S}^1):=
\left\{
\varphi\in
\mathcal C(\overline{\Omega}\times {\mathbb S}^1)\,\,
|\,\, \varphi(x,y)=\varphi(x,-y),\,\forall x\in\overline{\Omega},\,
\forall y\in {\mathbb S}^1\right\}
\]
is the space of continuous functions with \emph{antipodal} symmetry (the unit normal being well-defined $\mathcal H^1$-almost everywhere, up to the orientation, since $\Si$ in $\mathcal A_L(\Omega)$ is 1-rectifiable). Actually, \eqref{concrete} defines $\theta_\Sigma$ as a $1$-dimensional \emph{varifold}, i.e., 
a probability measure over $\overline{\Omega}\times \mathbb P^1$ with $\mathbb P^1$ the \emph{projective space}
(see \cite{alm,sim} for general treatises on the subject).
We denote by $\vari$ the space of $1$-dimensional varifolds with unit mass, that is,
probability measures over $\overline{\Omega}\times \mathbb P^1$, 
endowed with the usual weak-* topology.
Throughout, however, we shall always consider a  varifold $\theta\in\vari$
as
an {equivalence class} of probability measures on
$\overline{\Omega}\times  \mathbb{S}^1$, two measures being equivalent if and only if they
induce the same linear functional on  $\mathcal C_{\text{sym}}(\overline{\Omega}\times {\mathbb S}^1)$: thus,
by choosing a representative in its equivalence class,
we can still treat $\theta\in\vari$ as a
probability measure over $\overline{\Omega}\times {\mathbb S}^1$. With this agreement
the weak-* convergence
of a sequence $\{\theta_L\}$ in $\vari$ to a varifold $\theta\in\vari$, denoted by $\theta_L\weak \theta$,
takes the form
\begin{equation}\label{weak}
\lim_{L\to\infty}
\int_{\Ots} \varphi(x,y)
 \,d\theta_L
=
\int_{\Ots} \varphi(x,y)
\,d\theta\quad
\forall
\varphi\in\mathcal C_{\text{sym}}(\overline{\Omega}\times {\mathbb S}^1).
\end{equation}

Then, for every $L>0$ we define the functional
$F_L\colon\vari\to [0,\infty]$ 
\begin{equation}\label{functional}
F_L(\theta)=
\begin{cases}
{L^p}\,\sobn(\Omega,\Sigma) \quad &\text{if $\theta=\theta_\Si$ as in \eqref{concrete} for some $\Si\in\mathcal A_L(\Omega)$,}\\[2mm]
+\infty  &\text{otherwise}
\end{cases}
\end{equation} 
The scaling factor $L^p$ is natural: as we will see in the proof of the $\Gamma$-convergence result, by letting $L$ go to infinity, the minimum value in \eqref{problem} decreases as the power $L^p$. Of course, rescaling with this factor does not alter the original problem \eqref{problem} anyhow.
The asymmetric role played by $p$ and $q$ throughout the paper (for instance the fact that this scaling factor does not depend on $q$) is due to the different powers in the definition \eqref{PS}.

The definition of the $\Gamma$-limit functional $F_\infty$ is more involved.

\begin{definition}\label{d.varifold}
Given $\theta\in\vari$, by choosing a representative in its equivalence class we can
regard $\theta$ as an element of $\Prob(\cil)$,
and we can consider its \emph{first marginal} $\mu$, i.e., the probability measure
over $\overline\Omega$ defined by
\[
\mu(E):=\theta(E\times {\mathbb S}^1),\quad
\text{for every Borel set $E\subseteq\overline\Omega.$}
\]
Then, by disintegration theorems (see, e.g., \cite[Theorem~2.28]{amfupa} or \cite[Theorem~5.31]{amgisa}), one can \emph{disintegrate}
$\theta$ as $\mu\otimes\nu_x$, where $\{\nu_x\}$ is a
$\mu$-measurable
family of probability
measures over ${\mathbb S}^1$ defined for $\mu$-a.e. $x\in\overline\Omega$. This means
that
\begin{equation*}
\int_{\Ots} \varphi(x,y)d\theta=
\int_{\overline\Omega}
\left(
\int_\sfera \varphi(x,y)\,d\nu_x(y)\right)
\,d\mu(x)
,\quad
\forall\varphi\in C(\overline{\Omega}\times
 \mathbb{S}^1).
\end{equation*}
The measure $\mu$ depends only on $\theta$ as an element of
$\vari$ (not on the choice of its representative in $\Prob(\cil)$), while
the measures $\nu_x$ may well depend on the particular representative.
However, the integrals
\begin{equation*}
x\mapsto \int_\sfera |A(x) y\cdot y|^{\frac{1}{2}} \,d\nu_x(y)\quad
\text{(defined for $\mu$-a.e. $x\in\cOmega$)}
\end{equation*}
are independent of the particular representative of $\theta$, since the function
$(x,y)\mapsto |A(x) y\cdot y|^{1/2}$ belongs to $\mathcal C_{\text{sym}}(\overline{\Omega} \times {\mathbb S}^1)$. As a consequence,
if $\rho(x)$ denotes the density of $\mu$ (w.r.to the Lebesgue measure on $\overline\Omega$),
the functional
\begin{equation}\label{gammalimit}
\Fp(\theta):=
\displaystyle
\spq\bigg(\int_\Om \frac{f(x)^{\frac{p}{p-q}}}{\big(\rho(x)\int_{\mathbb{S}^1} |A(x)y\cdot y|^{\frac{1}{2}}\, d\nu_x(y)\big)^\frac{pq}{p-q}}\,dx\bigg)^\frac{p-q}{q}  
\end{equation}
is well defined for every $\theta\in\vari$ since it only depends on
$\theta$ as an element of $\vari$. Notice that $F_\infty$ has values in $[0,+\infty]$ and $F_\infty(\theta)=+\infty$ whenever $\rho\equiv 0$  almost everywhere on a set of positive measure of $\Omega$. Here the constant
\begin{equation}\label{spq}
{\spq}:= (pq+p-q)^{1-\frac{p}{q}}\frac{p^\frac{p}{q}}{(p-1)q}\left(2\int_0^1 \frac{ds}{(1-s^q)^\frac{1}{p}}\right)^{-p}
\end{equation}
is just the Poincar\'e-Sobolev constant for the unit interval (see Remark~\ref{rem.spq}).
\end{definition}

The main result of the paper is the following.
\begin{theorem}[$\Gamma$-convergence in the sub-homogeneous case]\label{t.main}
As $L\to+\infty$, the functionals $\Flp$ defined in \eqref{functional}
$\Gamma$-converge, with respect to the weak-*
topology on $\vari$, to the functional $\Fp\colon\vari \to [0,+\infty]$ defined in \eqref{gammalimit}.
\end{theorem}

When $q=1$ and $A$ is the identity matrix, an analogous $\Gamma$-convergence result, but in the restricted space of probability measures, has been obtained in the papers \cite{butsan, nay1, nay2, til} (where only the limit density of the minimizing sequences has been supplied). The proof of Theorem~\ref{t.main} inherits some of the ideas developed in \cite{tilzuc2}, where problem \eqref{problem} has been studied in the homogeneous linear case, corresponding to the exponents $p=q=2$. Notice that the extension to the sub-homogeneous case $p<q$ is not immediate anyhow: one has to face with the \emph{locality} of the functional $C(\Omega, \Sigma)$ (see Proposition~\ref{p.disconnected}), which requires several {ad hoc} arguments.

Now, as the space $\vari$ is compact in the weak-* topology, from $\Gamma$-convergence theory
(see~\cite{dalmaso} and also \cite[Section 5]{tilzuc3}) we can recover some information on the asymptotic behaviour of the minimizers of \eqref{problem}.

\begin{corollary}
For $L>0$, let $\Sigma_L$ be a minimizer of \eqref{problem} and let $\theta_{\Sigma_L}$ be the associated varifolds, according to \eqref{concrete}. Then, as $L\to+\infty$, $\theta_{\Sigma_L}\weak\theta_\infty$ (up to subsequences) in the weak-* topology of $\vari$, where $\theta_\infty$ is a minimizer of $F_\infty$. In particular the following facts hold.
\begin{itemize}
\item[-] \emph{(Limiting density of the Dirichlet region).}
For every square $Q\subset\Omega$,
\begin{equation*}
\lim_{L\to+\infty}\frac{\cH(\Sigma_L\cap Q)}{\cH(\Si_L)}=\int_Q \rho_\infty(x)\,dx\qquad
\end{equation*}
where $\rho_\infty$ is 
\begin{equation}\label{optden}
\rho_\infty(x):=\frac{f(x)^r/a_{\max}(x)^{rq/2}}{\int_{\Om}f(y)^r/a_{\max}(y)^{rq/2}\,dy}
\end{equation}
with $r:=p/(pq+p-q)$.

\item[-] \emph{(Local orientation of the Dirichlet region).} If $Q$ is contained in the anisotropy region $\{x\in \Omega: a_{\min}(x)<a_{\max}(x)\}$, then
for every $\psi\in C(\sfera)$ with $\psi(-y)=\psi(y)$,
\begin{equation}\label{optori}
\lim_{L\to+\infty}\frac 1  {\cH(Q\cap\Sigma_L)}
\int_{Q\cap\Sigma_L} \psi\bigl(\xi_L(x)\bigr) \,d\cH(x)
=
\dfrac{\int_Q \rho_\infty(x)\psi\bigl(\xi(x)\bigr)\,dx}
{\int_Q \rho_\infty(x)\,dx},
\end{equation}
where $\xi_L(x)$ is the unit normal to $\Sigma_L$ at $x\in\Sigma_L$, while $\xi(x)$ is
the (unique up
to the orientation) eigenvector of $A(x)$  relative to $a_{\max}(x)$.

\item[-] \emph{(Asymptotics of the Poincar\'e-Sobolev constant).}
\begin{equation}\label{optcost}
 \lim_{L\to +\infty} L^p \, \sobn(\Omega, \Sigma_L)=\spq  {\left(\int_{\Om}\frac{f(x)^r}{a_{\max}(x)^\frac{rq}{2}}\, dx\right)^{p+\frac{p}{q}-1}},
\end{equation}
with $r:=p/(pq+p-q)$.
\end{itemize}
\end{corollary}

The proof of Theorem~\ref{t.main} is constructive: oriented \emph{comb-shaped patterns} (see Definition~\ref{d.tile}), periodically reproduced inside $\Omega$ at the different scales \eqref{optden} and with the different orientations \eqref{optori} (so that they are locally orthogonal to eigenvectors relative to maximal eigenvalues), can be used to build examples of asymptotically optimal sets, i.e., sequences of sets satisfying~\eqref{optcost}. Of course, in the isotropy region $\{x\in \Omega: a_{\min}(x)=a_{\max}(x)\}$ the orientation becomes not relevant. Asymptotically optimal sets allows to explicitly compute the constant $\spq$ in the $\Gamma$-limit \eqref{gammalimit}.

\section{Some results on Poincar\'e-Sobolev constants}\label{sec2}

\subsection{General results}

The first result concerns the behaviour of Poincar\'e-Sobolev constants w.r.to set inclusions.

\begin{remark}[Monotonicity of Poincar\'e-Sobolev constants]\label{r.mon}

Let $\Omega_1, \Omega_2\subset \Omega$ be open sets with Lipschitz boundaries and $\Sigma_1, \Sigma_2\in \mathcal A_L(\Omega)$ for some $L>0$. If 
\[
\Omega_1\subset \Omega_2 \quad \text{and}\quad  \Sigma_1\supset (\Sigma_2\cap  {\Omega}_1)\cup  \partial \Omega_1,\footnote{For the monotonicity of Poincar\'e-Sobolev constants are sufficients the inclusions of the domain $\Omega_1$ in $\Omega_2$ and of the entire boundary $\partial \Omega_1$ plus the part of $\Sigma_2$ which intersect $\Omega_1$ in $\Sigma_1$, see Figure~\ref{fig.mon} for an example).}
\] 
then every function $u$ in $W^{1,p}(\Omega_1)$ such that $u= 0$ along $\Sigma_1$ also belongs to $W^{1,p}(\Omega_2)$ (by setting $u=0$ over $\Omega_2\setminus \overline \Omega_1$) and $u= 0$ along $\Sigma_2$. By the variational characterization \eqref{PS}, the class of admissible function for $\sobn(\Omega_2,\Sigma_2)$ is not smaller than the one for  $\sobn(\Omega_1,\Sigma_1)$, then
$$\sobn(\Omega_1,\Sigma_1)\leq \sobn(\Omega_2,\Sigma_2).$$
In particular, when $\Sigma_1=\partial \Omega_1$ and $\Sigma_2=\partial \Omega_2$, then the inclusion $\Omega_1\subset \Omega_2$ implies the inequality $\sobn(\Omega_1)\leq \sobn(\Omega_2)$.
\end{remark}

\begin{figure}
\centering
\subfigure{
\begin{tikzpicture}

  \draw[thick] (0,0) circle (1cm);
  \node at (1.2,0.5) {$\Omega_1$};
  \phantom{\draw (0,0) circle (2cm);
  \node at (2.3,0.3) {$\Omega_2$};}
  \draw[-,thick,black] (0,0)--(1,0);
  \draw[-,thick,black] (0,0)--(-0.3,-0.1);
  \draw[-,thick,black] (0,0)--(-0.3,0.1);
  \node at (0.1,-0.3)  {$\Sigma_1$};
  \end{tikzpicture}}
  \subfigure{
\begin{tikzpicture}

  \draw[dashed] (0,0) circle (1cm);
  \draw (0,0) circle (2cm);
  \node at (2.3,0.3) {$\Omega_2$};
  \draw[-,dashed,black] (0,0)--(1,0);
  \draw[-,dashed,black] (0,0)--(-0.3,-0.1);
  \draw[-,dashed,black] (0,0)--(-0.3,0.1);
  \node at (0.65,-0.25)  {$\Sigma_2$};
  \draw[-,thick,black] (0,0)--(1.5,0);
  \draw[-,thick,black] (0,0)--(1,0);
  \end{tikzpicture}}

 \caption{An example of domains $\Omega_1, \Omega_2$ (two circles of different radii) and sets $\Sigma_1,\Sigma_2$ that fulfill the assumptions for the validity of the monotonicity property: on the left $\Omega_1$ and $\Sigma_1$ on the right $\Omega_2$ and $\Sigma_2$ (in dashed line the set $\Sigma_1$).}\label{fig.mon}
\end{figure}
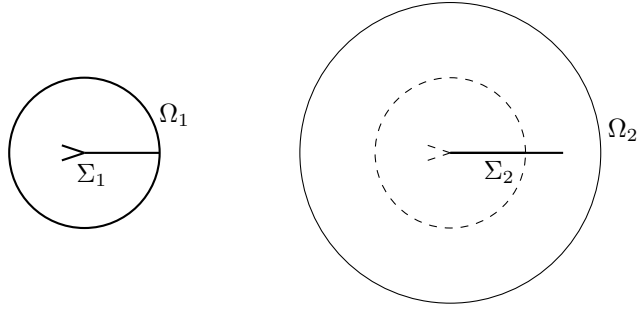

The next result, concerning the behavior of Poincar\'e-Sobolev constants on disconnected domains, will be responsible for the integral representation of the $\Gamma$-limit \eqref{gammalimit}. A very similar result is contained in \cite{brafra} (the case $q=1$ has been already proved in \cite{banwar}). 

\begin{proposition}[Poincar\'e-Sobolev constants of disconnected domains]\label{p.disconnected}
Let $\Omega_1$, $\Omega_2$$\subset \Omega$ be open sets such that $\Omega_1\cap \Omega_2\not= \emptyset$ and let $\Sigma\in \mathcal A_L(\Omega)$ for some $L>0$ such that $\Sigma\supset \partial \Omega_1\cup \partial\Omega_2$. Then
\begin{equation}\label{concomp1} 
\sobn(\Omega_1\cup \Omega_2,\Sigma)^\frac{q}{p-q}=\sobn(\Omega_1,\Sigma)^{\frac{q}{p-q}}+\sobn(\Omega_2,\Sigma)^{\frac{q}{p-q}}.
\end{equation}
In particular, when $\Sigma=\partial \Omega_1\cup \partial\Omega_2$ then $
\sobn(\Omega_1\cup \Omega_2)^\frac{q}{p-q}=\sobn(\Omega_1)^{\frac{q}{p-q}}+\sobn(\Omega_2)^{\frac{q}{p-q}}$.
\end{proposition}
\begin{proof}
Let
$u,u_1, u_2$ be maximizers corresponding to the constants $\sobn(\Omega_1\cup \Omega_2,\Sigma)$, $\sobn(\Omega_1,\Sigma)$ and $\sobn(\Omega_2,\Sigma)$. 

To prove the inequality $\leq$ in \eqref{concomp1},  let $\alpha$ be the constant $0\leq \alpha\leq 1$ defined by $\alpha:=(\int_{\Omega_1}f(x)u(x)^qdx)/(\int_{\Omega_1\cup \Omega_2}f(x)u(x)^qdx)$. 
Since, by \eqref{sobolev},
\[
\Big(\int_{\Omega_j} f(x)u(x)^qdx\Big)^{\frac pq} \leq  \sobn(\Omega_j,\Sigma)\int_{\Omega_j} |A(x)\nabla u(x)\cdot \nabla u(x)|^{\frac p2}dx
\]
for $j=1$ and $j=2$, then using the super-addivity of the power $\tau\mapsto\tau^{\frac pq}$ yields
\begin{equation}\label{c}
\frac{1}{\sobn(\Omega_1\cup\Omega_2,\Sigma)}\geq \frac{\alpha^{\frac pq}}{\sobn(\Omega_1,\Sigma)}+\frac{(1-\alpha)^{\frac pq}}{\sobn(\Omega_2,\Sigma)}\geq \min_{0\leq s\leq 1} \frac{s^{\frac pq}}{\sobn(\Omega_1,\Sigma)}+\frac{(1-s)^{\frac pq}}{\sobn(\Omega_2,\Sigma)}.
\end{equation}
By convexity of the power $\tau\mapsto\tau^{\frac pq}$, the minimum in the right-hand side of the previous inequality is reached by
\begin{equation}\label{alphat}
s=\frac{\sobn(\Omega_1,\Sigma)^\frac{q}{p-q}}{\sobn(\Omega_1,\Sigma)^\frac{q}{p-q}+\sobn(\Omega_2,\Sigma)^\frac{q}{p-q}},
\end{equation}
which plugged into \eqref{c} gives
\[
\frac{1}{\sobn(\Omega_1\cup\Omega_2,\Sigma)}\geq \frac{1}{\left(\sobn(\Omega_1,\Sigma)^\frac{q}{p-q}+\sobn(\Omega_2,\Sigma)^\frac{q}{p-q}\right)^\frac{p-q}{q}}.
\]
Therefore, raising to the power $q/(p-q)$ and passing to reciprocals one obtains the desired upper bound. 

For the reverse inequality $\geq$ in \eqref{concomp1}  given $t>0$ we test $\sobn(\Omega_1\cup \Omega_2, \Sigma)$ with the function $t^{\frac1q} u_1+u_2\in W^{1,p}(\Omega_1\cup\Omega_2)$, which is zero along $\Sigma$, to obtain
\[
\frac{1}{\sobn(\Omega_1\cup\Omega_2,\Sigma)}\leq  \frac{\alpha_t^{\frac pq}}{\sobn(\Omega_1,\Sigma)}+\frac{(1-\alpha_t)^{\frac pq}}{\sobn(\Omega_2,\Sigma)},
\]
where $\alpha_t:=(t\int_{\Omega_1}u_1(x)^q dx)/(t\int_{\Omega_1}u_1(x)^qdx+\int_{\Omega_2}u_2(x)^qdx)$.
Now, since $\alpha_t$ is continuous as a function of $t$, $\alpha_t\to 0$ as $t\to 0^+$ and $\alpha_t\to1$ as $t\to+\infty$, there exists a value $t$ such that $\alpha_t$ can be chosen as in \eqref{alphat}. Then plugging this value into the previous inequality provides also the upper bound in \eqref{concomp1}.
\end{proof}


In Proposition~\ref{p.disconnected} we tacitly used definition \eqref{PS} also in the case of irregular and disconnected sets.

\subsection{Results in the case of constant coefficients}

When the coefficient matrix $A(x)$ is constant, 
equal to a $2\times 2$ positive definite matrix 
independent of $x$, and $f=1$ it is possible to get further results.

\begin{remark}[Scaling of Poincar\'e-Sobolev constants]\label{r.scaling}
Let $\Sigma\in \mathcal A_L(\Omega)$ for some $L>0$ and let $t>0$. 
If the coefficient matrix $A$ is constant and $f= 1$ then by the change of variable $tx=y$ in the integrals of \eqref{PS} it follows that 
\begin{equation*}
\sobn(t\Omega, t \Sigma)=t^{p+\frac{2p}{q}-2}\sobn(\Omega,\Sigma).
\end{equation*}
\end{remark}

In the case of constant coefficients one can derive explicit formulas. 

\begin{remark}[Explicit Poincar\'e-Sobolev contants]\label{rem.spq}
The value $\spq$ appearing in \eqref{spq} is the Poincar\'e-Sobolev constant for the unit interval in one
dimension, namely
\begin{equation}\label{PS1}
\spq = \max_{u\in W^{1,p}_0(0,1) \setminus\{0\}} \frac
{\Big(\int_0^1 |u(z)|^q\,dz\Big)^{\frac pq}}{\int_0^1 |u'(z)|^p\,dz}.
\end{equation}
Notice that the minimum is attained by the first eigenfunction $u_1$, which does not change sign in $(0,1)$ and it is symmetric w.r.to $z=1/2$ with
\begin{equation}\label{one}
\spq=2^{\frac{p}{q}-1}\frac{\left(\int_0^{\frac12} |u_1(z)|^q dz\right)^{\frac{p}{q}}}{\int_0^{\frac{1}{2}} |u_1'(z)|^pdz}.
\end{equation}
Moreover, if $u_1$ is meant positive, then it is increasing and concave in $(0,1/2)$ (see \cite{draman, fralam} for more details).
Notice that the previous considerations on the constant $c_{p,q}$ holds true for the whole range $1\leq q<\infty$.

Now, in the specific case $q<p$ we can also say something for two-dimensional domains. If $A$ is the identity matrix and $f=1$, the Poincar\'e-Sobolev constant \eqref{PS} of an open rectangle $R\subset \Omega$ with Dirichlet condition prescribed along two parallel sides $l_1$ and $l_2$ at distance $h$ can be explicitly computed:
\[
C(R,l_1\cup l_2)=\spq |R|^{\frac{p}{q}-1}h^{p}.
\]
Indeed the estimate $\geq$ follows by testing \eqref{PS} with the one dimensional eigenfunction $u_1$, while the estimate $\leq$ by  Fubini's theorem, \eqref{PS1} and Jensen's inequality (a similar argument has been adopted for \cite[Equation (3.3)]{bra2}).
\end{remark}

When $A$ and $f$ are constants we may also bound $\sobn(\Omega,\Sigma)$ in
terms of some geometric quantities (see \cite{nay1,til, tilzuc1, tilzuc2} for similar bounds corresponding to the cases $q=1$ and $p=q$). 
The key point to derive lower and upper bounds is the linear change of variable $x=A^{1/2}y$ in the integrals of \eqref{PS}, which allows to reduce $\sobn(\Omega,\Sigma)$ to the Poincar\'e-Sobolev constant of a new domain associated to the identity matrix.
The following lower bound for Poincar\'e-Sobolev constants will be crucial for the $\Gamma$-liminf inequality \eqref{gammainf}, and thus is asymptotically optimal as $\cH(\Sigma)\to+\infty$.

\begin{theorem}[Lower bound for Poincar\'e-Sobolev constants]\label{t.lower}
Let $\Si\subset\overline\Omega$ be a compact set with $N$ connected component for some $N \in \N$. If the coefficient matrix $A$ is constant and $f= 1$ then
\begin{equation}\label{estb}
\sobn(\Omega,\Sigma)\geq \spq (2\delta)^{p+\frac{p}{q}-1} \frac{\cHa(\Si)^{\frac pq}}{\cHa(\Si)+N\pi \det A^{\frac12}\delta},
\end{equation} 
where
\begin{equation}\label{hausm}
\cHa(\Si):=\int_{\Si} |A\xi(x)\cdot\xi(x)|^\frac{1}{2}\,  d\cH(x),
\end{equation}
represents the ``Riemannian'' length of $\Si$ w.r.to $A$ with $\xi(x)$ (a measurable selection of) the unit \emph{normal} to $\Si$ at the point $x\in \Si$ (by rectifiability, this normal is well defined, up to the orientation, at $\cH$-a.e. $x\in \Si$) and
\begin{equation}\label{delta}
\delta:=\dfrac{|\Omega|}
 {\cHa(\Si)+\left(\cHa(\Si)^2+N\pi|\Omega|\det A^{\frac{1}{2}}\right)^{\frac12}}
\end{equation}
is the positive root of the quadratic equation $2\cH_A(\Sigma)\delta+N\pi\det A^{\frac{1}{2}}\delta^2=|\Omega|$.
\end{theorem}
\begin{proof}
Set $B = A^{-\frac{1}{2}}$ and change variable $y = Bx$ in the
two integrals of \eqref{PS} so that
\[
\sobn(\Omega,\Sigma)=\bigl(\det A^{\frac12}\bigr)^{\frac pq-1}\max_{v\in W_{B\Sigma}(B\Omega)\setminus\{0\}}\frac{\left(\int_{B\Omega} |v(y)|^qdy\right)^{\frac{p}{q}}}{\int_{B\Omega} |\nabla v(y)|^pdy},
\]
where $B\Omega$ and $B\Sigma$ are the images of $\Omega$ and $\Si$ through $B$. The above ratio is the same as the one considered in \cite{tilzuc1}, but with numerator and denominator's powers decoupled. We may follow the same trick and test with the function $$v(y)=u_1\left(\frac{\mathop{d_{B\Sigma}}(y)}{2\delta_B}\right),\quad y\in B\Om,$$ 
where $\mathop{d_{B\Sigma}}$ is the distance function to $B\Sigma$, $u_1$ is the solution to \eqref{PS1} and 
\begin{equation*}
\delta_B:=\dfrac{|B\Omega|}
 {\cH(B\Si)+({\cH(B\Si)^2+N\pi|B\Omega|})^\frac{1}{2}}
\end{equation*}
is the positive root of the quadratic equation $2\cH(B\Sigma)\delta+N\pi\delta^2=|B\Omega|$ (see also \cite[Equation~(2.2)]{tilzuc1}).
Notice that, by definition, $d_{B\Sigma}(y)\leq 2\delta_B$ for all $y\in B\Om$ and this guarantees that the argument of $u_1$ is always bounded from above by $1/2$.
Therefore, similarly to \cite[Equation (2.11)]{tilzuc1} (see also \cite[Remark 2.5]{tilzuc1}), we find that
\begin{equation}\label{last}
\sobn(\Omega,\Sigma)\geq \bigl(\det A^{\frac12}\bigr)^{\frac pq-1}\frac{(2\cH(B\Sigma))^\frac{p}{q}(2 \delta_B)^{p+\frac{p}{q}-1}}{2\cH(B\Sigma)+2N\pi \delta_B} \frac{\left(\int_0^{\frac12} u_1(t)^q dt\right)^{\frac{p}{q}}}{\int_0^{\frac{1}{2}} u_1'(t)^pdt}.
\end{equation}
As noticed in the proof of \cite[Theorem~2.1]{tilzuc2}, by the area formula we have $\cH(B\Sigma)=(\mathrm{det} A^{-\frac{1}{2}})\cHa(\Sigma)$, and moreover $|B\Omega|=(\mathrm{det} A^{-\frac12})|\Omega|$. Then $\delta_B=\delta$ as defined in \eqref{delta} and combined with \eqref{one} implies that \eqref{last} is equivalent to \eqref{estb}.
\end{proof}

The following upper bound for Poincar\'e-Sobolev constants of trapezoids will be at the basis of the construction of the recovering sequence, for the $\Gamma$-limsup inequality \eqref{gammasup2}. It turns out that it is asymptotically sharp whenever the height as well as the lateral sides (legs) of the trapezoid converge to zero. 

\begin{theorem}[Upper bound for Poincar\'e-Sobolev constants of trapezoids]\label{t.thin}
Let $E\subset\Omega$ be a trapezoid of height $h$ in the direction of the unit vector $\xi\in \mathbb R^2$, and with lateral sides (legs) of lengths $k_1$ and $k_2$.
If the coefficient matrix $A$ is constant and $f=1$, then
\begin{equation*}
\sobn(E)<\spq \left(|E|+\dfrac {\sqrt{a_{\max}}(k_1+k_2)h}{2|A\xi\cdot \xi|^{\frac{1}{2}}}\right)^{\frac pq-1}
\left(\frac{h}{|A\xi\cdot \xi|^\frac{1}{2}}\right)^{p}.
\end{equation*}

\end{theorem}
\begin{proof}
Let $\tau_1,\tau_2\in\mathbb R^2$ be the vectors that induce the lateral sides of $E$ of modulus $k_1,k_2$. Set
$B = A^{-\frac{1}{2}}$ and change variable $y = Bx$ in the
two integrals of \eqref{PS} so that
\[
\sobn(E)=\bigl(\det A^{\frac12}\bigr)^{\frac pq-1}\max_{u\in W_{B\partial E}(BE)\setminus\{0\}}\frac{\left(\int_{BE} |u(y)|^qdy\right)^{\frac{p}{q}}}{\int_{BE} |\nabla u(y)|^pdy},
\]
where $BE$ and $B\partial E$ are the images of $E$ and $\partial E$ through $B$. Since by definition $E$ is included in the strip $\{x\in\R^2\,|\,\, 0< x\cdot\xi <h\}$, then $BE$ is still a trapezoid, contained in the
strip $\{y\in\R^2\,|\,\, 0< y\cdot A^{1/2}\xi <h\}$,
so that its hight is given by
\begin{equation}
\label{newwidth}
h_B=\dfrac {h}{\left|
A^{1/2}\xi\right|}
=
\dfrac {h}{|A\xi\cdot \xi|^{\frac{1}{2}}},
\end{equation}
see Figure~\ref{fig.trapezoid}.

The previous maximum is the Poincar\'e-Sobolev constant associated to the identity matrix and can be estimated by considering the smallest open rectangle $R$ containing $BE$. Indeed, if $l_1, l_2$ are those sides of $R$ orthogonal to the vector $A^{1/2}\xi$, then 
by Remarks~\ref{r.mon} and \ref{rem.spq}  
we have
\[
\sobn(E)< \bigl(\det A^{\frac12}\bigr)^{\frac pq-1} \sobn(R,l_1\cup l_2)=
\spq\bigl(\det A^{\frac12}\bigr)^{\frac pq-1} |R|^{\frac{p}{q}-1}({h_B})^{p}.
\]
(where the strict inequality follows from the strict inclusion of the boundary conditions $l_1\cup l_2\subsetneq \partial R$).
Now, to estimate the area of $R$ we may write $|R|=|BE|+|R\setminus BE|$ and, by using $|BE|=(\mathrm{det} A^{-\frac12})|E|$, 
obtain
\begin{equation*}
\sobn(E)< \spq (|E|+\det A^{\frac12}|R\setminus BE|)^{\frac{p}{q}-1}({h_B})^{p}.
\end{equation*}
To prove the theorem it remains to estimate the area of $R\setminus BE$. Since this set is the union of two triangles we have
$$|R\setminus BE|< \frac{|B\tau_1|h_B}{2}+\frac{|B\tau_2|h_B}{2}\leq \dfrac {(k_1+k_2)h}{2\sqrt{a_{\min}}|A\xi\cdot \xi|^{\frac{1}{2}}},$$
where the last inequality is an equality whenever $\tau_1$ and $\tau_2$ are parallel to the eigenvector corresponding to the maximal eigenvalue $1/\sqrt{a_{\min}}$ of $B$. This conclude the proof.
\end{proof}

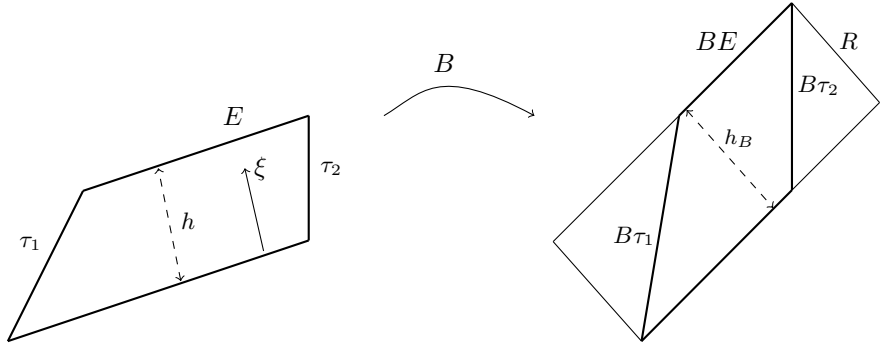
\begin{figure}
\centering
\subfigure{
\begin{tikzpicture}
  \draw[thick] (0,0)--(4,1.34);
  \draw[thick] (0,0)--(1,2);
  \draw[thick] (4,1.34)--(4,3);
  \draw[thick] (1,2)--(4,3);
  \node at (0.3,1.3) {\small $\tau_1$}; 
  \node at (4.3,2.3) {\small $\tau_2$}; 
  \node at (3,3)  {$E$};
  \draw[->] (5,3) .. controls (5.5,3.3) and (5.7,3.7) .. (7,3);
  \node at (5.8,3.7)  {$B$};
  \draw[<->, dashed] (2.3,0.8)--(2,2.3);
  \node at (2.4,1.6)  {\small $h$};
  \draw[->] (3.4,1.2) -- (3.15,2.3) node[right] {$\xi$};
  \end{tikzpicture}}
  \subfigure{
\begin{tikzpicture}[scale=0.5]
  \draw[thick] (0,0)--(4,4.02);
  \draw[thick] (0,0)--(1,6);
  \draw[thick] (4,4.02)--(4,9);
  \draw[thick] (1,6)--(4,9);
  \node at (2,8)  {$BE$};
  \draw[<->, dashed] (1.2,6.15)--(3.5,3.55);
  \node at (2.6,5.4)  {\footnotesize $h_B$};
  \draw[-,thin] (4,9)--(6.33,6.35);
  \draw[-,thin] (-2.35,2.65)--(0,0);
  \draw[-,thin] (-2.35,2.65)--(1,6);
  \draw[-,thin] (4,4.02)--(6.33,6.35);
  \node at (-0.2,2.8) {\small $B\tau_1$}; 
  \node at (4.7,6.8) {\small $B\tau_2$}; 
  \node at (5.5,8) {\small $R$}; 
  \end{tikzpicture}}
  \caption{A trapezoid and its image through the linear map $B$.}\label{fig.trapezoid}
 \end{figure}

\begin{remark}[An estimate with the area and the width]\label{rmkhkhkj}
It would be interesting to investigate whether it holds the stronger scale-invariant inequality:
\[
C(E)<c_{p,q} |E|^{\frac{p}{q}-1}\big(\text{width}(E)\big)^p,
\]
where $\text{width}(E)$ stands for the \emph{minimal width} of an arbitrary set $E$ 
(see \cite{tilzuc2} for a proof in the case $p=q=2$).
\end{remark}

We now recall the tile construction given in \cite{tilzuc2}.

\begin{definition}[Tile in the unit square]\label{d.tile}
Given $n\in\mathbb N$, let $\varepsilon_j>0$, $\beta_j>0$ and $\xi_j\in {\mathbb S}^1$ for all $1\leq j\leq n$. 
Slice the unit square $Y:=(0,1)\times (0,1)$ into $n$
stacked rectangles $Y_j$
($1\leq j\leq n$) of size $1\times h_j$,
the height $h_j$ being defined as
\begin{equation*}
h_j:=\dfrac{\beta_j \varepsilon_j
}{\sum_{j=1}^n \beta_j \varepsilon_j},
\end{equation*}
(clearly $\sum h_j=1$ so that the heights of
the $n$ rectangles match the height of $Y$).
Then, by drawing inside every $Y_j$ a maximal family of parallel line segments, orthogonal to $\xi_j$ and equally spaced a
distance  of $\eps_j$ apart from one another we further slice every rectangle $Y_j$ into several polygons
of width $\eps_j$. Let $K_j$ denote the union of all these line segments, orthogonal to $\xi_j$,
drawn inside $Y_j$. 

The \emph{tile} $T$ is then the compact and connected set defined as
\begin{equation*}
T:=R\cup S, \text{ with }
R:=\bigcup_{j=1}^n \partial Y_{j},\text{ and }\, S:=\bigcup_{j=1}^n K_{j}.
\end{equation*}
Notice also that $T\supset \partial Y$ since $\partial Y\subset R$ (see Figure~\ref{f.ex} for two examples).
\end{definition}

\begin{figure}
\centering

\subfigure
{
\begin{tikzpicture}[scale=0.45]
\rectangleA
\draw[->,thick,black] (6.6,7.9)--(5.6,8.5) node[left, below] {$\xi_1$} ; 
\draw[->,thick,black] (5,3.9)--(5.5,5) node[midway,right] {$\xi_2$};
\draw[->,thick,black] (4.8,0.7) --(5.7,1.5) node[right, below] {$\xi_3$} ;
\node at (-0.7,8) {$Y_1$} ;
\node at (-0.7,4.5) {$Y_2$} ;
\node at (-0.7,1) {$Y_3$} ;
\end{tikzpicture}
}
\quad
\subfigure
{
\begin{tikzpicture}[scale=0.45]

\rectangleB
\draw[->,thick,black] (2.9,8.8)--(3.9,9.5) node[right] {$\xi_1$} ; 
\draw[->,thick,black] (4.5,5.9)--(5.3,5.7) node[right] {$\xi_2$};
\draw[->,thick,black] (5,1.5) --(5,2.5) node[midway, left] {$\xi_3$} ;
\draw[->,thick,black] (4,0.5) --(5,0.5) node[right] {$\xi_4$} ;
\node at (-0.7,9) {$Y_1$} ;
\node at (-0.7,6) {$Y_2$} ;
\node at (-0.7,2.5) {$Y_3$} ;
\node at (-0.7,0.5) {$Y_4$} ;
\end{tikzpicture}
}
\caption{Examples of tiles in the unit square.}\label{f.ex}
\end{figure}
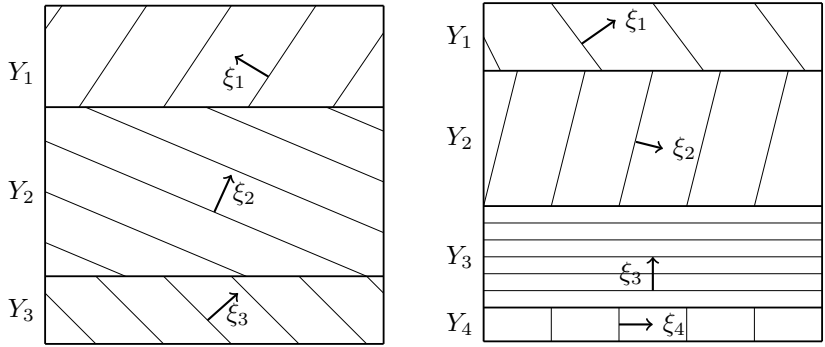

In the last result of this section we construct a fundamental pattern, whose
periodic homogenization inside $\Omega$ will be the main ingredient
in the construction of the recovery sequence, for the $\Gamma$-limsup inequality  \eqref{gammasup2}.

\begin{proposition}[Asymptotic bound for Poincar\'e-Sobolev constants]\label{p.tile}
Let $Q\subset~\R^2$ be an open
square with sides parallel to the coordinate axes, and
let $\nu$ be a probability measure over ${\mathbb S}^1$.
If the coefficient matrix $A$ is constant and $f=1$ then
for every length $\ell$ large enough, there exists a closed and connected set $\Sigma_\ell\in
{\mathcal A}_\ell(Q)$ with 
$\partial Q\subset\Sigma_\ell$ and
 $\lim_{\ell\to \infty}{\cH(\Sigma_\ell)}/{\ell}=1$ such that the varifolds associated with $\Sigma_\ell$ converge
to $|Q|^{-1}\chi_Q \otimes \nu$, that is, 
\begin{equation*}
\lim_{\ell\to\infty} \frac 1{\cH(\Sigma_\ell)}
\!\int_{\Sigma_\ell} \varphi(x,\xi_\ell(x))\,d\cH(x)\! = \!\frac
1 {|Q|}\!\int_Q\!\int_{{\mathbb
S}^1}\varphi(x,y)\,d\nu(y)\,dx,\forall\varphi\!\in\! C_\text{sym}(\overline{Q}\times {\mathbb S}^1)
\end{equation*}
where $\xi_\ell$ is any measurable selection of the unit normal to
$\Sigma_\ell$.
Moreover, 
\begin{equation}\label{stimabase}
 \limsup_{\ell\to \infty} \ell^p\sobn(Q,\Sigma_\ell)
\leq \frac{\spq |Q|^{{p}+\frac{p}{q}-1}}{\big(\int_{\mathbb{S}^1}
|Ay\cdot y|^{\frac12}\, d\nu(y) \big)^p}.
\end{equation}
\end{proposition}
\begin{proof} 
The sets $\Sigma_\ell$ will be obtained as the periodic homogenization,
inside the square $Q$, of suitably rescaled fundamental tiles $T_\eps$, for some small parameter $\varepsilon$,
initially constructed inside a unit square, see Definition~\ref{d.tile}. This construction was already employed in \cite[Proposition~4.1]{tilzuc2} where the weak-* convergence of the varifolds has been also proved. Therefore, we only have to prove \eqref{stimabase}.
We initially assume that the measure $\nu$ is \emph{purely atomic}, that is $\nu=\sum_{j=1}^n \beta_j \delta_{\xi_j}$
for suitable weights $\beta_j>0$ such that $\sum_{j=1}^n \beta_j=1$, $n\geq 1$ and unit vectors $\xi_j\in {\mathbb S}^1$ (here the number $\beta_j$ has the same meaning as in Definition~\ref{d.tile}).
With this notation, we have
\begin{equation*}
I:=\int_{{\mathbb S}^1}|A\xi\cdot \xi|^\frac12\,d\nu(\xi)=\sum_{j=1}^n \beta_j |A\xi_j\cdot \xi_j|^\frac12.
\end{equation*}
We consider the tile $T_\varepsilon$, according to Definition~\ref{d.tile}, where the values of the $\eps_j$s shall be fixed according to
\begin{equation}
\label{defepsj}
\eps_j:=\eps |A\xi_j\cdot \xi_j|^\frac{1}{2},\quad
1\leq j\leq n
\end{equation}
(where $\eps\ll 1$ is a scale parameter to be tuned later)
so that, by construction, every connected component of $Y\setminus T_\eps$ is, up to \emph{at most four} exceptions\footnote{Due to the corners of $Y_j$, a connected component may degenerate into a triangle, a pentagon and also into an hexagon (this is the case when the line segments are almost parallel to a diagonal of $Y_j$ and one polygon touches two opposite corners of $Y_j$).}, a trapezoid of height $\varepsilon_j$ in the direction $\xi_j$ for some $j\in\{1,\ldots,n\}$ and with lateral sides of lengths $c^1_j \varepsilon_j$ and $c^2_j  \varepsilon_j$ for suitable constants $c^1_j,c^2_j$ only depending on the direction $\xi_j$ (i.e., they are the tangents of the angle made up of $\xi$ and $\partial Y_i$). Notice also that for fixed $j$, the number of polygons (and thus of trapezoids) inside $Y_j$  is $O(1/\varepsilon_j)$. Then for a given $\gamma>0$ we can assume $\epsilon$ in \eqref{defepsj} so small so that the contributions of the constants on the exceptional connected components (i.e., those that are not trapezoids) are smaller than $\gamma$. These considerations, combined with \eqref{concomp1} and Theorem~\ref{t.thin}, imply
\begin{equation}
\label{ss6}
\begin{split}
\sobn(Y,T_\varepsilon)^\frac{q}{p-q}&<(\spq)^\frac{q}{p-q}
\sum_{j}\! \bigg(\frac{\varepsilon_j}{|A\xi_j\cdot \xi_j|^\frac{1}{2}}\bigg)^{\frac{pq}{p-q}}\bigg(\sum_i|E_{i,j}|+\frac{c_j\varepsilon_j}{|A\xi_j\cdot \xi_j|^\frac{1}{2}}\bigg)\!\!+\gamma\\
&<(\spq)^\frac{q}{p-q}\sum_{j} \bigg(\frac{\varepsilon_j}{|A\xi_j\cdot \xi_j|^\frac{1}{2}}\bigg)^{\frac{pq}{p-q}}\bigg(|Y_j|+\frac{c_j\varepsilon_j}{|A\xi_j\cdot \xi_j|^\frac{1}{2}}\bigg)\!+\gamma\\
&=(\spq)^\frac{q}{p-q}
\varepsilon^{\frac{pq}{p-q}}(1+nc_{\max}\varepsilon)+\gamma
\end{split}
\end{equation}
where we denoted by $\{E_{i,j}\}$ the family of trapezoids inside $Y_j$, $c_j$ a costant only depending on the matrix $A$ and on the direction of the vector $\xi_j$, $c_{\max}:=\max_j c_j$, and having used \eqref{defepsj} in the last passage. 

Now let $t$ denote the side length of the square $Q$. Given an integer $m\geq 1$ we may fit $m^2$ copies of the rescaled tile $(t/m)T_\eps$ inside $\overline Q$
as in an $m\times m$ checkerboard: the resulting tiling is then $1/m$--periodic in the two directions parallel to the sides of $Q$. We denote by $T_{\eps,m}$ the union of these $m^2$ rescaled tiles: this set is connected, because so is $T_\eps$
and,  by Definition~\ref{d.tile}, each tile shares a side of length $t/m$ with each neighbor.
The sets $\Sigma_\ell$ we want to construct are defined, for large $\ell$, as $\Sigma_\ell:=T_{\eps,m}$, with $\eps=\eps(\ell):=(t/\ell)^{\frac 2 3}$ and $m=m(\ell):=\lceil (\ell/t)^{\frac 1 3}I \rceil$,
(observe that $\eps\to 0$ and $m\to\infty$ when $\ell\to\infty$). Therefore, by construction, \eqref{concomp1} and Remark~\ref{r.scaling} we obtain
\[
\sobn(Q,\Sigma_\ell)^\frac{q}{p-q}=m^2\sobn\left(({1}/{m})Q,T_{\varepsilon,m}\right)^{\frac{q}{p-q}}=m^{-\frac{pq}{p-q}}t^{\frac{pq}{p-q}+2}\sobn(Y,T_\varepsilon)^{\frac{q}{p-q}}.
\]
From \eqref{ss6} we then obtain
\[
\sobn(Q,\Sigma_\ell)^\frac{q}{p-q}\leq (\spq)^\frac{q}{p-q}\frac{t^{2\frac{pq}{p-q}+2}}{(\ell I)^{\frac{pq}{p-q}}}(1+c_{\max}(t/\ell)^{2/3}) + \gamma,
\]
and raising to the power $(p-q)/q$, letting $\ell\to \infty$ then using the arbitrariness of $\gamma$, gives \eqref{stimabase} for purely atomic measures (recall that $t^2=|Q|$ is the area of the square $Q$). The inequality \eqref{stimabase} for arbitrary probability measures $\nu$ over ${\mathbb S}^1$ can be obtained via a standard diagonal argument  (see Step~3 in \cite[Proposition~4.1]{tilzuc2}).
\end{proof}

\section{The $\Gamma$-liminf inequality}\label{sec3}

This section is devoted to proving that the $\Gamma$-liminf functional is
not smaller than the functional $F_\infty$.
\begin{proposition}[$\Gamma$-liminf inequality]
For every  varifold $\theta\in\vari$ and every sequence $\{\theta_L\}\subset\vari$ such that $\theta_L\weak\theta$, it holds
\begin{equation}\label{gammainf}
  \liminf_{L\to+\infty} F_L(\theta_L)\geq F_\infty(\theta).
\end{equation}
\end{proposition}
\begin{proof}
Passing if necessary to a subsequence (not relabelled), we may assume that the liminf is a {finite} limit.  By \eqref{functional} this implies that, for $L$ large enough, every $\theta_L$ is equal to some $\theta_{\Sigma_L}$
as defined in \eqref{concrete}, for a suitable  $\Sigma_L\in\mathcal A_L(\Omega)$.
Therefore, by \eqref{concrete} and \eqref{weak}, the weak-* convergence $\theta_L\weak \theta$ takes the form
\begin{equation}\label{varifold}
\lim_{L\to+\infty}\frac{1}{\cH(\Si_L)}\int_{\Si_L} \varphi(x,\xi_{\Si_L}(x))d\cH(x)
=
\int_{\Ots} \varphi
\,d\theta\quad
\forall
\varphi\in\mathcal C_{\text{sym}}(\overline{\Omega}\times {\mathbb S}^1),
\end{equation}
where $\xi_{\Sigma_L}$ is the unit normal to $\Sigma_L$.
Similarly, using \eqref{functional} and \eqref{gammalimit}, the claim in \eqref{gammainf} is, after taking powers with $q/(p-q)$,
equivalent to 
\begin{equation}
\label{toprove}
\lim_{L\to +\infty}{L^{\frac{pq}{p-q}}}\sobn(\Omega,\Sigma_L)^\frac{q}{p-q}
\geq(\spq)^{\frac q{p-q}}\int_\Om \frac{f(x)^{\frac{p}{p-q}}}{(\rho(x)\int_{\mathbb{S}^1} |A(x)y\cdot y|^{\frac{1}{2}}\, d\nu_x(y) )^\frac{pq}{p-q}}dx
\end{equation}
where $\theta=\mu\otimes \nu_x$ is the slicing of $\theta$ and the function $\rho$ is the density of $\mu$ w.r.to the Lebesgue measure, see Definition~\ref{d.varifold}.

Now, to prove \eqref{toprove} fix a number $\eps>0$, an integer $k$ and a set $E\subset\Om$ which can be written as the interior of $\bigcup_j \overline{E_j}$, where the $E_j's$ are pairwise disjoint open squares of side-length $2^{-k}$, for some positive integer $k$. Note that one can choose $k$ arbitrarily large leaving $E$ unchanged (it suffices to decompose each $E_j$ into four equal squares of half the side-length, and repeat this procedure recursively).
For every $j$ consider the matrix $A_j:=A(x_j)$, where $x_j$ is the center of the square $E_j$:
since $A(x)$ is uniformly continuous, the conditions
\begin{equation}\label{ultima}
 |A(x)y\cdot y|  \leq \frac{1}{(1-\eps)}  |A_jy\cdot y| \qquad \forall x\in E_j, \, \forall
 y\in\mathbb{S}^1
\end{equation}
are satisfied as soon as $k$ is large enough (depending only on $\eps$). The finiteness of the
limit in \eqref{toprove} implies
that $\sobn(\Omega,\Sigma_L)\to 0$, and this in turn forces $\Sigma_L$ to converge to $\cOmega$ in the Hausdorff metric
(otherwise, a subsequence among the open sets $\Omega\setminus\Sigma_L$ would contain a ball
$B_r(x_L)$
of fixed
radius $r>0$, and by Remark~\ref{r.mon} we would
have $\sobn(\Omega,\Sigma_L)\geq \sobn(\Omega,\Sigma_L\cup \partial B_r(x_L))\geq  \sobn(B_r(x_L), \partial B_r(x_L))$: by \eqref{uniform}
this bound would be uniform in $L$, a contradiction). Since $\Sigma_L$ is a closed connected set, this also entails that (see, e.g., \cite{mostil})
\begin{equation}\label{intersection}
\lim_{L\to +\infty} \cH(\Si_L\cap E_j)=+\infty
\end{equation}
and moreover the set $\Sigma_L':=\bigcup_j(\Sigma_L\cap E_j)\cup \partial E_j$ 
is connected (if $L$ is large). 
Therefore, by Remark~\ref{r.mon} with $L\geq \cH(\Sigma_L)$ and \eqref{concomp1}, it follows that
\begin{equation*}
L^{\frac{pq}{p-q}}\sobn(\Omega,\Sigma_L)^\frac{q}{p-q}
\geq
\cH(\Sigma_L)^{\frac{pq}{p-q}}\sobn(E,\Sigma_L')^\frac{q}{p-q}
=
\cH(\Sigma_L)^{\frac{pq}{p-q}}\sum_{j}\sobn(E_j,\Si_L')^{\frac{q}{p-q}}
\end{equation*}
which, using \eqref{ultima} gives
\begin{equation}\label{dimliminf21}
L^{\frac{pq}{p-q}}\sobn(\Omega,\Sigma_L)^\frac{q}{p-q}
\geq (1-\varepsilon)
\sum_{j}(\inf_{E_j}f)^{\frac p{p-q}}{\cH(\Sigma_L)^{\frac{pq}{p-q}}}\sobn_j(E_j,\Si_L')^{\frac{q}{p-q}},
\end{equation}
where $C_j$ is the Poincar\'e-Sobolev constant \eqref{PS} associated to the constant coefficient matrix $A_j$ and density $f=1$, whenever $k$ is large enough (as we shall assume in the following). 

Inside each square $E_j$ by Theorem~\ref{t.lower} we have
\begin{equation}
\label{estb1}
\sobn_j(E_j,\Sigma_L')^{\frac{q}{p-q}}\geq (\spq)^{\frac{q}{p-q}}2^{\frac{pq}{p-q}+1} \frac{\ell_L^{\frac{p}{p-q}}\delta_L^{\frac{pq}{p-q}+1}}{(\ell_L+\pi \det A^{\frac12}\delta_L)^{\frac{q}{p-q}}},
\end{equation}
where $\ell_L$, according to \eqref{hausm}, is the Riemannian length
\begin{equation}
\label{deflL}
\ell_L=\int_{\Sigma_L'} |A_j\xi_L(x)\cdot \xi_L(x)|^\frac{1}{2}\,  d\cH(x),
\end{equation}
$\xi_L(x)$ being the unit normal to $\Sigma_L'$, while $\delta_L$, according to \eqref{delta}, is
\begin{equation}\label{deftL}
 \delta_L=\dfrac{|E_j|}{\ell_L +(\ell_L^2+\pi |E_j|\det {A_j}^{\frac{1}{2}})^\frac{1}{2}}.
\end{equation}
We shall let
$L\to+\infty$ in \eqref{dimliminf21} and use \eqref{estb1},  hence
we are interested in
the asymptotics (as $L\to+\infty$)
of both $\delta_L$ and $\ell_L$. As $A_j=A(x_j)$, \eqref{deflL} and \eqref{uniform}
give
\[
\ell_L\geq
\kappa_0^\frac{1}{2}\cH\bigl(\Sigma_L'\bigr)
>
\kappa_0^\frac{1}{2}\cH(\Sigma_L\cap E_j),
\]
so that $\ell_L\to+\infty$  by \eqref{intersection} and from \eqref{deftL} we find the asymptotics
\begin{equation}
\label{asyt}
 \delta_L
 \sim \dfrac{|E_j|}{2\ell_L }\quad\text{as $L\to+\infty$.}
 \end{equation}
Moreover, since the contribution of $\partial E_j$ to the integral in \eqref{deflL} is fixed,
while that of $\Sigma_L\cap E_j$ is dominant by \eqref{intersection}, it follows that
\begin{equation}
\label{asylL}
\ell_L\sim \int_{\Sigma_L\cap Q} |A_j \xi_{\Sigma_L}(x)\cdot \xi_{\Sigma_L}(x)|^{\frac{1}{2}}\,  d\cH(x)\quad
\text{as $L\to+\infty$.}
\end{equation}
Using \eqref{estb1} with \eqref{asyt}, we obtain
\begin{equation}\label{eq00}
\liminf_{L\to+\infty}\cH(\Sigma_L)^{\frac{pq}{p-q}}\sobn_j(\Si_L', E_j)^{\frac{q}{p-q}}\!\!\geq\! \!(\spq)^{\frac{q}{p-q}}\!\liminf_{L\to\infty}\left(
\frac{\cH(\Si_L)}{\ell_L}\right)^{\frac{pq}{p-q}}\!\!|E_j|^{\frac{pq}{p-q}+1}.
\end{equation}
By \eqref{asylL}, using \eqref{varifold} with 
$\varphi(x,y)=\eta_j(x) |A_j y\cdot y|^\frac12$ for all $x\in\cOmega$, $y\in\sfera$ and some  cutoff function $\eta_j\in C(\cOmega)$ such that $0\leq \eta_j\leq 1$ and $\eta_j\equiv 1$ over $\overline{E_j}$, we infer that
\[
\begin{split}
\limsup_{L\to+\infty}
\frac
{\ell_L} {\cH(\Si_L) }
&\leq
\lim_{L\to+\infty}
\frac
1 {\cH(\Si_L) }
\int_{\Sigma_L} \eta_j(x)|A_j \xi_{\Sigma_L}(x)\cdot \xi_{\Sigma_L}(x)|^{\frac12}\,  d\cH(x)\\
&=
\int_{\Ots} \eta_j(x)|A_j y\cdot y|^{\frac12}
\,d\theta.
\end{split}
\]
It is possible to let $\eta_j(x)\downarrow \chi_{\overline E_j}(x)$ pointwise in the last integral,
and obtain by dominated convergence
\begin{equation}\label{eq3}
\begin{split}
\limsup_{L\to+\infty}
\frac
{\ell_L} {\cH(\Si_L) }
&\leq
\int_{\overline E_j\times \sfera} |A_j y\cdot y|^\frac12
\,d\theta\\
&=
\int_{\overline E_j}\left( \int_{\sfera} |A_j y\cdot y|^\frac12\,d\nu_x(y)
\right)
\,d\mu(x)+\varepsilon |E_j|,
\end{split}
\end{equation}
where 
$\varepsilon |E_j|$ shall serve to avoid vanishing quantities at denominator (see below). 

Therefore, letting $L\to+\infty$ in \eqref{dimliminf21}, by combining \eqref{eq00} with \eqref{eq3},  we obtain
\begin{equation}\label{fine}
\lim_{L\to+\infty} F_L(\theta_L)^\frac{q}{p-q}=\lim_{L\to+\infty} L^{\frac{pq}{p-q}}\sobn(\Omega,\Sigma_L)^\frac{q}{p-q}\geq (1-\varepsilon) \int_E m_k(x) dx
\end{equation}
where $m_k$ is the piecewise constant function defined on $E$ as follows:
\[
m_k(x)=(c_{p,q})^\frac{q}{p-q} \sum_{j}(\inf_{E_j}f)^{\frac p{p-q}}\left(\frac{|E_j|}{\int_{\overline E_j}\int_{\sfera} |A_j y\cdot y|^\frac12\,d\nu_x(y)
\,d\mu(x)+\varepsilon|E_j|}\right)^{\frac{pq}{p-q}}\!\!\chi_{E_j}(x).
\]
Since by Radon-Nykodim theorem 
\[
\lim_{k\to\infty} m_k(x)=(\spq)^\frac{q}{p-q} {f(x)^{\frac p{p-q}}}\left(\frac{1}{\rho(x)\int_{\sfera} |A(x) y\cdot y|^\frac{1}{2}\,d\nu_x(y)+\varepsilon}\right)^{\frac{pq}{p-q}} \text{for a.e. $x\in E$,}
\]
letting $k\to\infty$ in \eqref{fine} we obtain from the Fatou lemma that
\[
\lim_{L\to+\infty} F_L(\theta_L)^\frac{q}{p-q}\geq (1-\varepsilon)(\spq)^\frac{q}{p-q} \int_E \frac{f(x)^{\frac{p}{p-q}}}{(\rho(x)\int_{\mathbb{S}^1} |A(x)y\cdot y|^\frac12\, d\nu_x(y)+\varepsilon )^\frac{pq}{p-q}}dx.
\]
Letting $E\uparrow \Omega$ and $\varepsilon\downarrow 0^+$ we finally get \eqref{toprove}, thanks to monotone convergence.
\end{proof}

\section{The $\Gamma$-limsup inequality}\label{sec4}

This section is devoted to proving that the $\Gamma$-limsup functional is
not greater than the functional $F_\infty$. We need to introduce the class of measures given in \cite{tilzuc2}.

\begin{definition}\label{d.step}
For $t>0$, let $\Qt$ denote the collection of all those open squares $Q_i\subset\R^2$, with side-length $t$
and corners
on the lattice $(t{\mathbb Z})^2$, such that $Q_i\cap \Omega\not=\emptyset$.
We say that a varifold  $\theta\in\vari$ is \emph{fitted to $\Qt$} if
it can be represented as
\begin{equation}\label{mufit}
\theta=\sum_{Q_i\in \Qt} \rho_i \chi_{\Omega\cap Q_i} \otimes \nu_i
\end{equation}
for suitable constants $\rho_i\geq 0$ and probability measures
$\nu_i$ over ${\mathbb S^1}$, satisfying
\begin{equation}\label{constants}
\sum_{Q_i\in\Qt} \rho_i|\Omega\cap Q_i| = 1.
\end{equation}
\end{definition}

\begin{proposition}[$\Gamma$-limsup inequality]\label{proptile2}
For every  varifold $\theta\in\vari$, there exists
a sequence of varifolds $\{\theta_L\}\subset\vari$ such that $\theta_L\weak\theta$ and,
moreover,
\begin{equation}\label{gammasup2}
  \limsup_{L\to \infty} F_L(\theta_L)\leq F_\infty(\theta).
\end{equation}
\end{proposition}
\begin{proof}
We divide the proof into two steps.

\smallskip
- \emph{Step 1: $\Gamma$-convergence for measures fitted to $\Qt$.}
In this step we also assume that $\theta\in\vari$ is a varifold \emph{fitted to $\Qt$}. Then we claim that there exists a sequence of continua
$\Sigma_L\in \mathcal A_L(\Omega)$ such that \eqref{varifold} holds true
and
\begin{equation}\label{gammasup}
 \limsup_{L\to \infty}L^p\sobn(\Omega,\Sigma_L)\leq
F_\infty(\theta).
\end{equation}
We keep for $\theta$ the same notation as in Definition~\ref{d.step} and fix $\eta>1$.
By replacing $t$ with $t/2^n$ for some large $n\geq 1$ and relabelling the $\rho_i$s
(thus keeping $\nu$ fitted to $\Qt$), we may assume that $t$ is so small that the following two conditions hold:
\begin{itemize}
\item[(H1)] no connected component of $\partial\Omega$ is strictly contained in any square $Q_i\in\Qt$;
\item[(H2)] for every square $Q_i\in\Qt$, there exists a positive definite matrix $A_i$ and a constant $f_i>0$ such that for every $x\in \Omega\cap Q_i$ and $y\in\mathbb S_1$
\begin{equation}\label{assu2}
\frac{1}{\eta^{\frac{1}{p}}} | A_iy\cdot y| \leq |A(x) y\cdot y|\leq \eta^{\frac{1}{p}} | A_iy\cdot y| \  \text{ and } \ \frac{1}{\eta^\frac{q}{2p}} f_i \leq f(x) \leq \eta^\frac{q}{2p} f_i.
\end{equation}
\end{itemize}
Condition (H1) holds, for $t$ small enough, since $\partial \Omega$, being Lipschitz,
has finitely many connected components whose diameters have a positive lower bound.
On the other hand, (H2) is guaranteed, for small $t$, by the uniform continuity of the functions $A(x)$ and $f(x)$
(one can set, e.g.,  $A_i=A(x_i)$ and $f_i=f(x_i)$, for some $x_i\in \Omega\cap Q_i$).

The sets $\Sigma_L$ we want to construct will be obtained as a patchwork of sets $\Sigma^i_\ell$,
one for each square $Q_i\in\Qt$, obtained from Proposition~\ref{p.tile}. More precisely,
for every square $Q_i\in\Qt$ we apply Proposition~\ref{p.tile} to the square $Q_i$,
the measure $\nu_i$ and the matrix $A_i$. This yields, for
large enough $\ell$, sets $\Sigma^i_\ell$ with $\partial Q_i\subset\Sigma^i_\ell$ and $\cH(\Sigma^i_\ell)\sim\ell$ as $\ell\to\infty$ such that for every $\varphi\in C_\text{sym}(\overline{Q_i}\times {\mathbb S}^1)$
\begin{equation}
\label{varifoldsconverge4} \lim_{\ell\to\infty} \frac 1{\cH(\Sigma^i_\ell)}
\int_{\Sigma^i_\ell} \varphi(x,\xi^i_\ell(x))\,d\cH(x) = \frac
1 {|Q_i|}\int_{Q_i}\int_{{\mathbb
S}^1}\varphi(x,y)\,d\nu_i(y)\,dx
\end{equation}
where $\xi^i_\ell$ is any measurable selection of the unit normal to
$\Sigma^i_\ell$,
and moreover,
\begin{equation}\label{ii}
\displaystyle \limsup_{\ell\to \infty} \ell^p\sobn_i(Q_i,\Sigma_\ell^i)
\leq \frac{\spq (f_i)^{\frac{p}{q}} |Q_i|^{{p}+\frac{p}{q}-1}}{\big(\int_{\mathbb{S}^1}
|A_iy\cdot y|^{\frac12}\, d\nu_i(y) \big)^p}
\end{equation}
where $C_i$ is the Poincar\'e-Sobolev constant associated to the constant coefficients $A_i$ and $f_i$.

Observe that the domain $\Omega$ plays no role in this construction, and this is natural
for those squares $Q_i\in\Qt$ such that $Q_i\subset\Omega$. If, however, $Q_i\in\Qt$
is such that $Q_i\cap\partial\Omega\not=\emptyset$, the weak-* convergence \eqref{varifoldsconverge4}
(that occurs in the whole $\overline{Q_i}\times {\mathbb S}^1$)
can still
be \emph{localized} to $\overline{\Omega\cap Q_i}\times {\mathbb S}^1$. More precisely,
since $\partial(\Omega\cap Q_i)$ has null Lebesgue measure, and
the limit measure $|Q_i|^{-1}\chi_{Q_i}\otimes \nu_i$ does not
charge the cylinder $\partial(\Omega\cap Q_i)\times {\mathbb S}^1$, from \eqref{varifoldsconverge4}
we infer that for every $\varphi\in C_\text{sym}(\overline{\Omega\cap Q_i}\times {\mathbb S}^1)$
\begin{equation*}
\lim_{\ell\to\infty} \frac 1{\cH(\Sigma^i_\ell)}
\!\int_{\overline{\Omega}\cap \Sigma^i_\ell} \varphi(x,\xi^i_\ell(x))d\cH(x) \!=\! \frac
1 {|Q_i|}\int_{\Omega\cap Q_i}\!\int_{{\mathbb
S}^1}\!\varphi(x,y)d\nu_i(y)dx,
\end{equation*}
(in the first integral, a localization to $\overline{Q_i}$ is implicit
since $\Sigma^i_\ell\subset \overline{Q_i}$ by assumption).

In \eqref{mufit} we may assume that $\rho_i>0$ for all $i$, since if $\rho_i=0$ for some
$i$ then the right-hand side of \eqref{gammasup} is $+\infty$ and the inequality is trivial. Then, for large $L$
we can define the sets 
$
\Sigma_L:= \bigcup_{Q_i\in \Qt} (\Sigma^i_{t^2\rho_i L}\cap \Omega)\cup\partial\Omega,
$
where $\Sigma^i_{t^2\rho_iL}$ denotes the set $\Sigma^i_{\ell}$ defined above, when $\ell=t^2\rho_i L$
(recall that $t^2=|Q_i|$ is the area of $Q_i$). Actually, the sets $\Sigma_L$ just constructed are not necessarily in $\mathcal A_L(\Omega)$ (because we do not know wheter $\cH(\Sigma_L)\leq L$ for every $L$ sufficiently large) but only asymptotically equivalent to $L$, namely $\lim_{L\to \infty}{\cH(\Sigma_L)}/L=1$. Anyhow, this drawback can easily be removed as in \cite[Lemma~4.4 and Remark~4.5]{tilzuc2}. 
Now, to prove \eqref{gammasup}, observe that every connected component
of $\Omega\setminus \Sigma_L$ is contained inside some square $Q_j\in\Qt$ and
moreover, by construction, $\Sigma_L\supset \partial(\Omega\cap Q_i)$ for \emph{every}
square $Q_i\in\Qt$. As a consequence, Proposition~\ref{p.disconnected} applies and we have
\footnote{The last two inequalities follow from \eqref{assu2} and Remark~\ref{r.mon}, respectively: observe that, since
$A_i$ is a constant matrix and $b_i$ is a constant, the Poincar\'e-Sobolev constant 
$\sobn_{i}(Q_i,\Sigma^i_{t^2\rho_i L})$
makes sense for every $Q_i$, while $\sobn(Q_i,\Sigma^i_{t^2\rho_i L})$ would make
no sense if $Q_i\cap\partial\Omega\not=\emptyset$, since $A(x)$ is defined
only for $x\in\overline{\Omega}$.}
\[
\begin{split}
\sobn(\Omega,\Sigma_L)^\frac{q}{p-q}&=
\sum_{Q_i\in\Qt}
\sobn(\Omega\cap Q_i,\Sigma_L)^\frac{q}{p-q}
= \sum_{Q_i\in\Qt} \sobn(\Omega\cap Q_i,\Sigma^i_{t^2\rho_i L})^\frac{q}{p-q}\\
&\leq  \eta^\frac{q}{p-q} \sum_{Q_i\in\Qt}
\sobn_i(\Omega\cap Q_i,\Sigma^i_{t^2\rho_i L})^\frac{q}{p-q}
\leq \eta^\frac{q}{p-q} \sum_{Q_i\in\Qt}
\sobn_i(Q_i,\Sigma^i_{t^2\rho_i L})^\frac{q}{p-q}.
\end{split}
\]
Then, swapping $\limsup$ and the sum,
and using \eqref{ii}
with $\ell=t^2\rho_i L$, we obtain
\[
\begin{split}
\limsup_{L\to\infty} L^\frac{pq}{p-q}\sobn(\Omega,\Sigma_L)^\frac{q}{p-q}
&\leq
\eta^\frac{q}{p-q}
\limsup_{L\to\infty}
\sum_{Q_i\in\Qt}
L^\frac{pq}{p-q}\sobn_i(Q_i,\Sigma^i_{t^2\rho_i L})^\frac{q}{p-q}\\
&\leq
\eta^\frac{q}{p-q}
\sum_{Q_i\in\Qt}
\limsup_{L\to\infty}
L^\frac{pq}{p-q}\sobn_i(Q_i,\Sigma^i_{t^2\rho_i L})^\frac{q}{p-q}\\
&\leq 
 (\eta \cdot c_{p,q})^\frac{q}{p-q}
\sum_{Q_i\in\Qt}
\frac{(f_i)^{\frac{p}{p-q}}|Q|^{\frac{pq}{p-q}+1}}{\big(t^2\rho_i\int_{\mathbb{S}^1}
|A_i y\cdot y|^\frac12\, d\nu_i(y) \big)^\frac{pq}{p-q}}\\
&\leq 
 (\eta \cdot c_{p,q})^\frac{q}{p-q}
\sum_{Q_i\in\Qt}
\frac{(f_i)^{\frac{p}{p-q}}|Q|}{\big(\rho_i\int_{\mathbb{S}^1}
|A_i y\cdot y|^\frac12\, d\nu_i(y) \big)^\frac{pq}{p-q}}.
\end{split}
\]
The claim \eqref{gammasup} then follows from \eqref{assu2}, up to a multiplicative factor depending on $\eta>1$, as $\theta$ is as in \eqref{mufit} and the functional $F_\infty$ defined in \eqref{gammalimit} takes the form
\begin{equation*}
\Fp(\theta)^{\frac{q}{p-q}}=
(\spq)^{\frac{q}{p-q}}\sum_{Q_i\in\Qt}\int_{\Om\cap Q_i} \frac{f(x)^{\frac{p}{p-q}}}{\big(\rho_i\int_{\mathbb{S}^1} |A(x)y\cdot y|^{\frac{1}{2}}\, d\nu_i(y)\big)^{\frac{pq}{p-q}}}dx.
\end{equation*}
The constant $\eta$ can be remove by a diagonal argument (see again Step~3 in \cite[Proposition~4.1]{tilzuc2}). 

\smallskip

- \emph{Step 2: density in energy in the class of measures fitted to $\Qt$.}
The passage to a generic varifold $\theta$ is a standard argument in $\Gamma$-convergence theory (see \cite{dalmaso}). It suffices to prove that 
for every varifold $\theta\in\vari$ there exist varifolds $\theta_L\in\vari$, each fitted to $\Qt$ for some $t>0$ that may
depend on $L$, such that
\begin{equation}
\label{weak1}
\theta_L\weak \theta \quad\text{in $\vari$, as $L\to\infty$}
\end{equation}
and, at the same time,
\begin{equation}
\label{densen}
\limsup_{L\to\infty} F_\infty(\theta_L)\leq F_\infty(\theta).
\end{equation}
The approximations are as in \cite{tilzuc2}. Consider any family of Borel sets $\Omega_i$ such that $
\overline{\Omega}=\bigcup_i \Omega_i$, 
$\Omega\cap Q_i\subseteq\Omega_i \subseteq
\overline{\Omega\cap Q_i}$, and $
\Omega_i\cap\Omega_j=\emptyset$, for all $i\not=j$,
and define
\begin{equation*}
\theta_L:=\sum_{Q_i\in \Qt} \rho_i \chi_{\Omega\cap Q_i} \otimes \nu_i,
\end{equation*}
where the constants $\rho_i$ and the probability measures $\nu_i$ are given by
\[
\rho_i:= \frac{\theta(\Omega_i \times {\mathbb S}^1)}{|\Omega_i|}
\quad
\text{and}
\quad \nu_i(E):=
\frac{\theta(\Omega_i\times E)}
{\theta(\Omega_i \times {\mathbb S}^1)}
\quad
\text{(for every Borel set $E\subseteq {\mathbb S}^1$).}
\] 
We only focus on \eqref{densen} (assuming all $\rho_i>0$, otherwise
\eqref{densen} is trivial as already observed). Let $\theta=\mu\otimes \nu_x$ be the disintegration of $\theta$ (as discussed in Definition~\ref{d.varifold}),
and let $\rho\in L^1(\Omega)$ be the density of $\mu$ with respect to the Lebesgue measure
(clearly, $\mu\geq \rho$ as measures).
For every $Q_i\in\Qt$, and for every $x\in \Omega\cap Q_i$, we have
\[
\begin{split}
\rho_i \int_{\mathbb{S}^1}
|A(x)y\cdot y|^\frac{1}{2}\, d\nu_i(y)
=&
\frac1 {|\Omega_i|}
\int_{\Omega_i\times {\mathbb S}^1}
|A(x)y\cdot y|^\frac{1}{2}  \,d\theta(z,y)
\\
=&\frac1 {|\Omega_i|}
\int_{\Omega_i} \left(\int_{{\mathbb S}^1}
|A(x)y\cdot y|^\frac{1}{2}
\,d\nu_z(y)\right)
d\mu(z)\\
\geq
&\frac 1{|\Omega_i|}
\int_{\Omega_i} \left(\int_{{\mathbb S}^1}
|A(x)y\cdot y|^\frac{1}{2}
\,d\nu_z(y)\right) \rho(z)\,dz.
\end{split}
\]
Now, for every $\eta>1$, since the matrices $A(x)$ are uniformly continuous and positive
definite over $\overline{\Omega}$, if $L$ is large enough (and consequently
the side length $t=1/L$ of every $Q_j\in\Qt$ is small enough) we have
\[
x,z\in \Omega_i\Rightarrow
|A(x)y\cdot y|^\frac{1}{2}
\geq \frac{1}{\eta}|A(z)y\cdot y|^\frac{1}{2}, \quad f(x)\leq \eta^q f(z), \quad
\forall y\in{\mathbb S}^1.
\]
This allows us to replace, up to a factor $\eta$, $A(x)$ with $A(z)$ in the previous estimate:
thus we obtain
\[
\frac{f(x)^{\frac{1}{q}}}{\rho_i \int_{\mathbb{S}^1}
|A(x)y\cdot y|^\frac{1}{2}\, d\nu_i(y)}
\leq\eta^2
\frac{f(z)^{\frac{1}{q}}}{\frac {1} {|\Omega_i|}
\int_{\Omega_i}  \rho(z)\big(\int_{{\mathbb S}^1}
|A(z)y\cdot y|^\frac{1}{2}
d\nu_z(y)\big)
dz}
\]
valid for $L$ large enough (depending only on $\eps$).
Taking the powers with $pq/(p-q)$ and letting $L\to \infty$, by Radon-Nykodim theorem, we obtain that
\[
\limsup_{L\to+\infty}\frac{f(x)^{\frac{p}{p-q}}}{\big(\rho_i \int_{\mathbb{S}^1}
|A(x)y\cdot y|^\frac{1}{2}\, d\nu_i(y)\big)^\frac{pq}{p-q}}
\leq \eta^{\frac{2pq}{p-q}} \frac{f(z)^{\frac{p}{p-q}}}{\big(\rho(z)\int_{{\mathbb S}^1} |A(z)y\cdot y|^\frac{1}{2} d\nu_z(y) \, dz\big)^\frac{pq}{p-q}}
\]
for a.e. $z\in \Omega$. Finally, by integration over $\Omega$, the finiteness of $F_\infty(\mu)$ and dominated convergence one obtains \eqref{gammasup2} (actually after taking powers with $q/(p-q)$ and multiplying  with $c_{p,q}$)
up to a multiplicative factor depending on  $\eta$. By the arbitrariness of $\eta$ the $\Gamma$-limsup inequality is then proved.
\end{proof}

\section{The case $q\geq p$.}\label{sec5}
 
The Poincar\'e-Sobolev inequality \eqref{sobolev} holds for other pairs of exponents $p, q$: in fact given $1< p< +\infty$ the exponent $q$ may be chosen up to the \emph{conjugate exponent} $p^*$ of $p$, defined, when $p<2$, as $p^*:=2p/(2-p )$ and, when $p\geq 2$, as $p^*:=+\infty$. Then one may wonder whether the results proved in this paper can be further extended to the case $q\geq p$.
The main difference with the sub-homogeneous case lies in the fact that, when $q\geq p$, the constant $C(\Omega,\Sigma)$ is no longer local, i.e., 
\begin{equation}\label{ff}
\text{if }\Sigma=\partial \Omega_1\cup \partial\Omega_2 \text{ then } C(\Omega_1\cup\Omega_2)=\max\{C(\Omega_1),C(\Omega_2)\}
\end{equation}
(cf. with Proposition~\ref{p.disconnected}). 
However, Remarks~\ref{r.mon} and \ref{r.scaling} and also Theorem~\ref{t.lower} remain true, also for $q\geq p$. We have to distinguish the cases $p=q$ and $q> p$.

In the \emph{homogeneous case} $p=q$ the explicit formula given in Remark~\ref{rem.spq} for Poincar\'e-Sobolev constants with mixed boundary conditions is still true; this implies the validity of Theorem~\ref{t.thin}. Consequently, with only minor changes  
from the linear case $p=q=2$, previously treated in \cite{tilzuc1,tilzuc2} (see also \cite{henzuc,tilzuc3} for other related results), one can prove the following result. 

\begin{theorem}[$\Gamma$-convergence in the homogeneous case]
As $L\to +\infty$ the functionals $\Flp$ defined in \eqref{functional} $\Gamma$-converge, with respect to the weak-* topology on $\vari$, to the following functional
\begin{equation}\label{ff2}
c_{p,p}
\esssup_{x\in\Omega}\frac{f(x)}{\Big(\rho(x)\int_{\mathbb{S}^1} |A(x)y\cdot y|^\frac{1}{2}\, d\nu_x(y)\Big)^p}.
\end{equation}
\end{theorem}

Instead, in the \emph{super-homogeneous case} $q>p$, the situation extremely changes and an extra difficulty emerges: the minimum value of \eqref{problem} stops to decay as $1/L^p$ (and the lower bound provided in Theorem~\ref{t.lower} is no longer sharp). In fact, one can prove that the minimum value decreases as the rescaling power $1/L^{p+2(\frac{p}{q}-1)}$ as $L\to+\infty$. Moreover, also an estimate as in Remark~\ref{rmkhkhkj} can not be true for $q>p$. It is sufficient to take $E=(0,1)\times (-L,L)$, and then let $L$ go to $+\infty$, by observing that 
$$C(E_L)\to C\left((0,1)\times\mathbb R\right)>0 \qquad\text{and}\qquad |E_L|^{\frac{p}{q}-1}\to 0.$$
In the end, an estimate as the one in Remark~\ref{rmkhkhkj} is the deep reason at the basis of the different scaling and limiting behaviour of these Poincar\'e-Sobolev constants.

In view of these considerations we have the following.

\begin{openproblem}[$\Gamma$-convergence in the super-homogeneous case]
Find the $\Gamma$-limit of the functional $G_L\colon \vari\to [0,\infty]$ defined by
\begin{equation*}
G_L(\theta)=
\begin{cases}
{L^{p+2(\frac{p}{q}-1)}}\,\sobn(\Omega,\Sigma) \quad &\text{if $\theta=\theta_\Si$ as in \eqref{concrete} for some $\Si\in\mathcal A_L(\Omega)$,}\\[2mm]
+\infty  &\text{otherwise,}
\end{cases}
\end{equation*}  
and characterize the asymptotically optimal configurations.
\end{openproblem}

In view of \eqref{ff} we still expect a $\Gamma$-limit with a supremal structure as in \eqref{ff2} and comb-shaped configurations to be asymptotically optimal. Anyhow, a proof for this evidences seems quite challenging.

\subsection*{Acknowledgements}
An anonymous referee who carefully read the paper and pointed out interesting issues is kindly acknowledged.

This work is part of the Research Project INdAM for Young Researchers (Starting Grant) \emph{Optimal Shapes in Boundary Value Problems} and of the INdAM - GNAMPA Project 2018 \emph{Ottimizzazione Geometrica e Spettrale}.


\end{document}